\documentclass[12pt]{amsart}
\usepackage{amsmath}
\usepackage{amssymb}
\usepackage{epsfig}
\usepackage{graphicx}
\numberwithin{equation}{section}

\input xy
\xyoption{all}

\newcommand{\bC}{\mathbb C}

\newcommand{\bN}{\mathbb N}
\newcommand{\bP}{\mathbb P}
\newcommand{\bQ}{\mathbb Q}
\newcommand{\bR}{\mathbb R}
\newcommand{\bZ}{\mathbb Z}

\newcommand{\cA}{\mathcal A}
\newcommand{\cB}{\mathcal B}

\newcommand{\cE}{\mathcal E}

\newcommand{\cG}{\mathcal G}
\newcommand{\cH}{\mathcal H}
\newcommand{\cI}{\mathcal I}

\newcommand{\cO}{\mathcal O}

\newcommand{\dv}{\mathrm{dv}}

\newcommand{\ra}{\rightarrow}

\newcommand{\Pic}{\operatorname{Pic}}

\newcommand{\Hom}{\operatorname{Hom}}
\newcommand{\Ext}{\operatorname{Ext}}
\newcommand{\wLambda}{\widetilde{\Lambda}}
\newcommand{\Aut}{\mathrm{Aut}}

\theoremstyle{plain}
\newtheorem{prop}{Proposition}

\newtheorem{theo}[prop]{Theorem}

\newtheorem{lemm}[prop]{Lemma}

\theoremstyle{definition}
\newtheorem{defi}[prop]{Definition}

\newtheorem{ques}[prop]{Question}

\newtheorem{rema}[prop]{Remark}

\newtheorem{exam}[prop]{Example}

\makeatother
\makeatletter

\author{Brendan Hassett}
\address{Department of Mathematics\\
Rice University, MS 136 \\
Houston, Texas  77251-1892 \\
USA}
\email{hassett@rice.edu}

\author{Yuri Tschinkel}
\address{Courant Institute\\
                New York University \\
                New York, NY 10012 \\
                USA }
\email{tschinkel@cims.nyu.edu}

\address{Simons Foundation\\
160 Fifth Avenue\\
New York, NY 10010\\
USA}

\title[Extremal rays and automorphisms]{Extremal rays and automorphisms of holomorphic symplectic varieties}

\begin{document}
\date{\today}

\maketitle

\section{Introduction}
\label{sect:intro}

For last fifteen years, numerous authors have studied the birational geometry
of projective irreducible
holomorphic symplectic varieties $X$, seeking to relate extremal contractions $X\ra X'$
to properties of the Hodge structures on $H^2(X,\bZ)$ and $H_2(X,\bZ)$, regarded as lattices
under the Beauville-Bogomolov form. Significant contributions have been made by
Huybrechts, Markman, O'Grady, Verbitsky, and many others \cite{HuyBasic},  \cite{MarkJAG},
\cite{ogrady}, \cite{verb}, see also \cite{huy}.

The introduction of Bridgeland stability conditions by Bayer and Macr\`i provided
a conceptual framework for understanding birational contractions and their
centers \cite{BM2,BM1}. In particular, one obtains a transparent classification
of extremal birational contractions, up to the action of monodromy,
for varieties of K3 type \cite{BHT}.

Here we elaborate the Bayer-Macr\`i machinery through concrete examples and 
applications. We start by stating the key theorem and organizing the
resulting extremal rays in lattice-theoretic terms; see Sections~\ref{sect:generalities}
and \ref{sect:formal}. We describe exceptional loci in small-dimensional cases in 
Sections~\ref{sect:loci} and \ref{section:EOR}. Finding concrete examples
for each ray in the classification can be computationally involved; we 
provide a general mechanism for writing down Hilbert schemes with prescribed
contractions in Section~\ref{sect:OER}. Then we turn to applications.
Section~\ref{sect:auto} addresses a question of Oguiso and 
Sarti on automorphisms of Hilbert schemes.
Finally, we show that the ample cone of a polarized variety 
$(X,h)$ of K3 type cannot
be read off from the Hodge structure on $H^2(X,\bZ)$ in Section~\ref{sect:ambiguity};
this resolves a question of Huybrechts.

\

{\bf Acknowledgments}: 
The first author was supported by NSF grants 1148609 and 1401764.
The second author was supported by NSF grant 1160859.
We are grateful to B.~Bakker, D.~Huybrechts, E.~Macr\`i, and A.~Sarti for helpful
discussions;
the manuscript benefited from suggestions by K.~Hulek.
G.~Mongardi has informed us of joint work with Knutsen and Lelli-Chiesa 
addressing closely related questions.
We especially thank A.~Bayer for allowing us to use material arising
out of our collaboration and explaining subtle aspects of his work with Macr\`i.

\section{Recollection of general theorems}
\label{sect:generalities}

Let $X$ be deformation equivalent to the Hilbert scheme of length $n\ge 2$
subschemes of a K3 surface. 
Markman \cite[Cor.~9.5]{MarkSurv} describes an extension of lattices 
$
H^2(X,\bZ) \subset \wLambda
$
and weight-two Hodge structures
$
H^2(X,\bC) \subset \wLambda_{\bC}
$
characterized as follows:  
\begin{itemize}
\item{the orthogonal complement of $H^2(X, \bZ)$ in $\wLambda$ has rank one, and is generated by
a primitive vector of square $2n-2$;}
\item{as a lattice
$\wLambda \simeq U^4 \oplus (-E_8)^2$
where $U$ is the hyperbolic lattice and $E_8$ is the positive definite
lattice associated with the corresponding Dynkin diagram;}
\item{there is a natural extension of the monodromy action on $H^2(X,\bZ)$ 
to $\wLambda$; the induced action on $\wLambda/H^2(X,\bZ)$
is encoded by a character $cov$ (see \cite[Sec.~4.1]{MarkJAG});}
\item{we have the following Torelli-type statement:  $X_1$ and $X_2$ are
birational if and only if there is Hodge isometry
$$\wLambda_1 \simeq \wLambda_2$$
taking $H^2(X_1,\bZ)$ isomorphically to $H^2(X_2,\bZ)$;}
\item{if $X$ is a moduli space $M_v(S)$ of sheaves over a K3 surface $S$
with Mukai vector $v$ then there is an isomorphism from $\wLambda$ to the
Mukai lattice of $S$ taking 
$H^2(X,\bZ)$ to $v^{\perp}$.}
\end{itemize}
Generally, we use $v$ to denote a primitive generator for the orthogonal complement
of $H^2(X,\bZ)$ in $\wLambda$.  Note that $v^2=\left(v,v\right) = 2n-2$.
When $X \simeq M_v(S)$ we may take the Mukai vector $v$ as the generator.  

\begin{exam} \label{exam:hilb}
Suppose that $X=S^{[n]}$ for a K3 surface $S$ so that
$$\wLambda = U \oplus H^2(S,\bZ)$$
with $v$ in the first summand. Then we can write
$$
H^2(S^{[n]},\bZ)=\bZ \delta \oplus H^2(S,\bZ)
$$
where $\delta$ generates $v^{\perp} \subset U$ and
satisfies $\left(\delta,\delta\right)=-2(n-1)$.
\end{exam}

There is a canonical homomorphism
$$\theta^{\vee}:\wLambda \twoheadrightarrow H_2(X,\bZ)$$
which restricts to an inclusion
$$H^2(X,\bZ) \subset H_2(X,\bZ)$$
of finite index.  
By extension, it induces a $\bQ$-valued Beauville-Bogomolov form on
$H_2(X, \bZ)$.

\begin{exam} \label{exam:hilb2}
Retaining the notation of Example~\ref{exam:hilb}: 
Let $\delta^{\vee} \in H_2(X,\bZ)$ be the class orthogonal
to $H^2(S,\bZ)$ such that $\delta \cdot \delta^{\vee}=-1$.
We have $\theta^{\vee}(\delta)=2(n-1)\delta^{\vee}$.
\end{exam}

Assume $X$ is projective.  
Let $H^2(X)_{alg} \subset H^2(X,\bZ)$ and $\wLambda_{alg}\subset \wLambda$ denote
the algebraic classes, i.e., the integral classes of type $(1,1)$.  
The Beauville-Bogomolov form on $H^2(X)_{alg}$ has signature $(1,\rho(X)-1)$,
where $\rho(X)=\dim(H^2_{alg}(X)).$
The {\em Mori cone} of $X$ is defined as the closed cone in $H_2(X,\bR)_{alg}$
containing the classes of algebraic curves in $X$.  The {\em positive cone} (or more accurately,
non-negative cone) in
$H^2(X,\bR)_{alg}$ is the closure of the connected component of the cone
$$\{D \in H^2(X,\bR)_{alg}: D^2 > 0 \}$$
containing an ample class. 
The dual of the positive cone in $H^2(X,\bR)_{alg}$ is the 
positive cone.

\begin{theo} \label{theo:main} \cite{BHT}
Let $(X,h)$ be a polarized holomorphic symplectic manifold as above.
The Mori cone in $H_2(X,\bR)_{alg}$ is generated by classes
in the positive cone 
and the images under $\theta^{\vee}$ of the following:
\begin{equation} \label{eqn:theta}
\{a \in \wLambda_{alg}: a^2 \ge -2, \left| \left(a,v\right)\right| \le v^2/2, \left(h,a\right)>0\}.
\end{equation}
\end{theo}

\section{Formal remarks on Theorem~\ref{theo:main}}
\label{sect:formal}

\begin{enumerate}
\item{For $a$ as enumerated in (\ref{eqn:theta}) write $R:=\theta^{\vee}(a) \in H_2(X,\bZ)$.
Recall that $\left(R,R\right)<0$ and $R$ is extremal in the cone described in Theorem~\ref{theo:main}
if and only if $R$
generates the extremal ray of the birational contraction $X \ra X'$ 
associated with the corresponding wall \cite[\S 5,12]{BM2}.}
\item{As discussed in \cite[Th.~12.1]{BM2}, the walls in Theorem~\ref{theo:main}
also admit a natural one-to-one correspondence with
\begin{equation} \label{eqn:alt}
\{\hat{a} \in \wLambda_{alg}: \hat{a}^2 \ge -2, 0 \le \left(\hat{a},v\right) \le v^2/2,
\left(h,\hat{a}\right)>0 \text{ if } \left(\hat{a},v\right)=0 \}.
\end{equation}
Indeed, in cases of (\ref{eqn:theta}) where $\left(a,v\right)<0$ we take
$\hat{a}=-a$. From now on, we utilize these representatives of the walls.}
\item{The saturation $\cH$ of the
lattice $\left<v,a\right>$ is the fundamental invariant
of each case. Observe that $\cH$ has signature $(1,1)$ if and only
if $\left(R,R\right)<0$. It is possible for $\cH \supsetneq \left<v,a\right>$;
however, in small dimensions we can express
$\cH=\left<v,a'\right>$ for some other $a'$ satisfying the conditions
in (\ref{eqn:theta}).} 
\item{Suppose $\cH$ has signature $(1,1)$.
Since $h$ is a polarization on $X$, we have
$$
\left(h,h\right)>0, \quad \left(h,v\right)=0
$$
and 
$\left<h,a,v\right>$ is a lattice of signature $(2,1)$.}
\item{$\cH$ has signature $(1,1)$ 
if and only if 
$$\left(a,a\right) \left(v,v\right) < \left(a,v\right)^2.$$
This is automatic if $\left(a,a\right) = -2$, or $\left(a,a\right)=0$ and
$\left(a,v\right) \neq 0$.  
Since $\left(a,v\right) \le \left(v,v\right)/2$ we necessarily have
\begin{equation} \label{eqn:aabound}
\left(a,a\right) < \left(v,v\right)/4,
\end{equation}
and
\begin{equation} \label{eqn:vminusabound}
\left(v-a,v-a\right) \ge \left(a,a\right) \ge -2.
\end{equation}
Moreover, we also find
\begin{equation} \label{eqn:mixedbound}
\left(a,v-a\right)\ge 1.
\end{equation}
If $\left(a,a\right)< 0$ this follows from $\left(a,v\right)\ge 0$.
If $\left(a,a\right)=0$ we deduce $\left(a,v-a\right)\ge 0$
but equality would violate our assumption on the signature of $\cH$.
If $\left(a,a\right) >0$ then
$$\left(a,v\right) > \sqrt{\left(a,a\right) \left(v,v\right)} >2 \left(a,a\right).
$$
Moreover, equality holds in (\ref{eqn:mixedbound}) precisely when
\begin{equation} \label{eqn:HCcase}
\left(a,a\right)=0, \quad \left(v,a\right)=1.
\end{equation}
}
\end{enumerate}
We shall rely on these observations in Section~\ref{section:EOR}
to streamline our enumeration of cases.

\section{Description of the exceptional loci}
\label{sect:loci}
We describe the exceptional loci $E$ of extremal contractions $X\ra X'$
mentioned in Section~\ref{sect:formal}, up to birational equivalence.
Our analysis follows
\cite{BM2,BM1} provided $X=M_v(T)$ for some K3 surface $T$,
or a moduli space of twisted sheaves over $X$.
We expect this is valid generally, however.
Indeed, generically this follows formally from the monodromy classification
of extremal rays of \cite{BHT}.

Let $v$ be a Mukai vector with $\left(v,v\right)>0$
and fix $\cH\ni v$ of signature $(1,1)$ as in Section~\ref{sect:formal}.
We use bounds (\ref{eqn:aabound}), (\ref{eqn:vminusabound}), and 
(\ref{eqn:mixedbound}) freely.

We define the {\em effective classes} of $\cH$ to be the monoid generated by $0$ and
the $D\in \cH$ satisfying $\left(D,D\right) \ge -2$ and $\left(v,D\right)>0$.
If $\cH$ admits a spherical class $s$ (with $\left(s,s\right)=-2$) orthogonal
to $v$ then we take one of $\{s,-s\}$ to be effective,
the one meeting the polarization $h$ positively.
This differs from the definition of \cite[5.5]{BM2} in that it does
not depend on the choice of a stability condition.  
An {\em irreducible} spherical class is one that is
indecomposable in the effective monoid. 
An effective element $D$ is {\em reducible} if one of the following holds:
\begin{itemize}
\item{$D$ is spherical but not irreducible;}
\item{$D$ is not spherical and there exists an irreducible spherical class $s$ such that $\left(s,D\right)<0$;}
\item{$D$ is isotropic but not primitive.}
\end{itemize}
Every effective $D$ with $\left(D,D\right)<-2$ is necessarily reducible.  
\begin{rema}
Suppose $\cH=\Pic(S)$ for a K3 surface $S$ and $v$ is big and semiample on $S$. Then the 
reducible effective classes $D$ are those for which $|mD|$ admits no irreducible
divisors for any $m>0$.  
\end{rema}

\begin{defi}
A {\em Hilbert-Chow decomposition} of $v$ is an expression
$$v=a+b \quad a,b \in \cH \text{ effective},$$
where $\left(a,a\right)=0$ and $\left(v,a\right)=1$.
An {\em irreducible decomposition} takes the form
$$v=a + b, \quad a,b \in \cH \text{ irreducible effective.}$$
Hilbert-Chow and irreducible decompositions are collectively called
{\em basic decompositions}.
A basic decomposition is {\em distinguished} if $\cH=\left<a,b\right>$.
We do not keep track of the order of $a$ and $b$. 
\end{defi}
Note that either $\left(h,a\right)>0$ or $\left(h,b\right)>0$; moreover,
$\left(v,a\right)$ or $\left(v,b\right)$ is $\le \left(v,v\right)$;
finally, $\left(a,a\right),\left(b,b\right) \ge -2$.
Thus basic decompositions 
are instances of the walls indexed by 
(\ref{eqn:theta}) and (\ref{eqn:alt}) and having negative Mukai form.
Recall these index extremal rays of birational contractions.

Each $\cH$ associated with an extremal ray gives rise to a distinguished
decomposition; some $\cH$ admit multiple basic decompositions.
We will explain below why the isotropic vectors require extra care.
\begin{ques} \label{ques:organize}
Let $\cH$ arise from an extremal ray of $X \ra X'$ as above.
Is there a bijection between basic decompositions of $v$
and irreducible components of the exceptional locus 
of the contraction?
\end{ques}
Bayer and Macr\`i \cite[\S 14]{BM2} describe a more encompassing correspondence
between strata of the exceptional locus
and decompositions of $v$, under the
assumption that $\cH$ admits no isotropic or spherical 
vectors, i.e., elements $w$ with $\left(w,w\right)=0,-2$.
This assumption is never satisfied in small dimensions, however.
Based on this evidence and the examples we have
computed, we assume Question~\ref{ques:organize} has a positive
answer in the analysis below.

Let $M_a$ denote the moduli space of stable objects of type $a$
and $M_{v-a}$ the moduli space of stable objects of type $v-a$.
We only care about these up to birational equivalence so we
need not specify the precise stability condition.
A typical element of $E$ corresponds to an extension
$$0 \ra \cA \ra \cE \ra \cB \ra 0$$
where $\cA \in M_a,\cB \in M_{v-a},$ and $\cE$ represents
an element of $\bP(\Ext^1(\cB,\cA))$. (Generally, one has exact triangles rather than
extensions but our goal here is only to sketch representative examples for 
each monodromy orbit.)  These have expected dimension
$$\dim(M_a) + \dim(M_{v-a}) + \dim(\Ext^1(\cB,\cA)))-1=
\left(v,v\right) - \left(a,v-a\right) + 3$$
which equals
$$\dim(X)+1-\left(a,v-a\right),$$
i.e., the expected codimension of $E$ is
\begin{equation} \label{eqn:expcodim}
\left(a,v-a\right)-1\ge 0.
\end{equation}
When there is strict equality the geometry can be encapsulated by the diagram:
$$\begin{array}{ccccc}
\bP^{\left(a,v-a\right)-1} & \longrightarrow  & E & \longrightarrow & M_v(T) \\
      &     & \downarrow &   &  \\
      &     & M_a(T)\times M_{v-a}(T) &   & 
\end{array}
$$

\subsection*{Hilbert-Chow contractions}
The case $\left(\ref{eqn:HCcase}\right)$ requires additional explanation;
this is the `Hilbert-Chow' case of \cite[\S 10]{BM1}, \cite[\S 5]{BM2}:
For concreteness, take $X=M_v(T)$ with Mukai vector $v=(1,0,1-n)$, where 
$\dim(X)=2n$; set $a=(0,0,-1)$ so that $v-a=(1,0,2-n)$.
Thus 
$$
M_v(T)=T^{[n]}, \quad M_{v-a}(T)=T^{[n-1]},
$$ 
and $M_a(T)$ parametrizes
shifted point sheaves $\cO_p[-1]$.
Given distinct $p_1,\ldots,p_n \in T$, the natural inclusion of
ideal sheaves gives an exact sequence
$$0 \ra \cI_{p_1,\ldots,p_n} \ra \cI_{p_1,\ldots,p_{n-1}} \ra \cI_{p_1,\ldots,p_{n-1}}|p_n\simeq \cO_{p_n} \ra 0$$
and thus an exact triangle
$$ \cO_{p_n}[-1] \ra \cI_{p_1,\ldots,p_n} \ra \cI_{p_1,\ldots,p_{n-1}}.$$
This reflects the fact that the vector space
$$\Hom(\cI_{p_1,\ldots,p_{n-1}},\cO_{p_n})=
\Ext^1(\cI_{p_1,\ldots,p_{n-1}},\cO_{p_n}[-1])\simeq \bC.$$
Now suppose that $p_{n-1}=p_n$; then
$$\Hom(\cI_{p_1,\ldots,p_{n-1}},\cO_{p_{n-1}})= T_{p_{n-1}}\simeq \bC^2$$
which means that $E$ is birationally a $\bP^1$-bundle over 
$$\Delta=\{(\Sigma,p): p\in \Sigma \} \subset T^{[n-1]}\times T\simeq M_{v-a}(T)\times M_a(T).$$
In particular, the exceptional locus in the Hilbert-Chow case is irreducible.

\subsection*{Isotropic vectors}
Suppose we have a decomposition of the form
$$v=\underbrace{a + \cdots + a}_{N \text{ times}} + b, \quad N\ge 1,$$
where $a$ and $b$ are primitive and $a$ is isotropic. We continue to assume
that $\left(v,a\right),\left(v,b\right) \ge 0$ and $\left(b,b\right) \ge -2$.
We analyze strata in the exceptional locus associated with such
decompositions.

Consider the moduli space $M_{Na}(T)$ for an appropriate generic stability
condition \cite[Thm.~2.15]{BM2}. Since $a$ is isotropic and the generic point 
of $M_a(T)$ parametrizes simple objects, we have
$$M_{Na}(T) = \underbrace{M_a(T) \times \cdots \times M_a(T)}_{N \text{ times }}$$
which has dimension $2N$. Moreover, $M_b(T)$ has dimension $\left(b,b\right)+2 \ge 0$
and the generic point is stable. Given $\cA \in M_a(T)$ and $\cB \in M_b(T)$
the non-split extensions
$$ 0 \ra \cA \ra \cE \ra \cB \ra 0$$
are parametrized by a projective space of dimension 
$$\left(a,b\right)-1=\left(a,v\right)-1.$$
Now consider those of the form
$$ 0 \ra \cA^{\oplus N} \ra \cE' \ra \cB \ra 0,$$
where we assume the restriction to each summand $\cA$ is non-trivial.
The isomorphism classes of $\cE'$ that arise in this way---neglecting the extension
data---are parametrized by 
$$\underbrace{\bP^{\left(a,v\right)-1} \times \cdots \times \bP^{\left(a,v\right)-1}}_{N \text{ times }}.$$
Putting everything together, the expected codimension of the corresponding stratum is
$$N(\left(v,a\right)-1),$$
which is typically larger than the codimension (\ref{eqn:expcodim}) of the stratum
associated with the decomposition
$$v=a + (v-a).$$
Note that when 
$\left(v,a\right)=1$ we are in the Hilbert-Chow case discussed above.

Thus decompositions involving isotropic vectors with multiplicities correspond to non-maximal
strata.

\section{Enumeration of rays}
\label{section:EOR}
In this section, for each monodromy orbit of extremal rays we describe the geometry
of the exceptional locus of the associated contraction.
This completes the analysis started in \cite{HT4} by employing the recent work of Bayer
and Macr\`i \cite{BM2,BM1}.
We organize the information first by dimension 
(or equivalently, by $\left(v,v\right)$) and then by the magnitude
of $\left(v,a\right)$.  
Such explicit descriptions have been used in connection with the following problems:
\begin{itemize}
\item{constructing explicit Azumaya algebras realizing 
transcendental Brauer-Manin obstructions to weak approximation and the
Hasse principle \cite{HVAV,HVA};}
\item{modular constructions of isogenies between K3 surfaces and
interpretation of moduli spaces of K3 surface with level structure
\cite{McKSaTaVA};}
\item{explicit descriptions of derived equivalences among K3 surfaces and perhaps
varieties of K3 type;}
\item{analysis of birational and biregular automorphisms of holomorphic symplectic
varieties, see e.g. \cite{HTJus,BCNWS}.}
\end{itemize}
For example, when we have an exceptional divisor of the form
$$\begin{array}{ccc}
\bP^{r-1} & \ra & E \\
	  &     & \downarrow \\
 	  &     & S \times M
\end{array}$$
where $S$ is a K3 surface and $M$ is holomorphic symplectic (perhaps a point!),
we may interpret $M$ as a parameter space of
Brauer-Severi varieties over $S$. These naturally defined families can be quite
useful for arithmetic applications.

\

We write $S^{[m]}$ as shorthand for the deformation equivalence class 
of the Hilbert scheme of a K3 surface.  Note that the notation
$S\times S^{[2]}$ just means a product of a K3 surface and a manifold
`Discriminant' refers to the lattice $\left<v,a\right>$.
We note cases where there are inclusions
$$\cH=\left<v,a\right> \supsetneq \left<v',a'\right>=\cH'$$
as then the exceptional locus associated to $\cH$ may be reducible,
as noted in Question~\ref{ques:organize}.
Here we analyze whether these arise from basic decompositions of $v$.

\subsection{Dimension four} The case $\left(v,v\right)=2$
has been explored in \cite{HT2}.

\begin{center}
\begin{tabular}{|c|c|c|l|l|}
\hline
$\left(a,a\right)$ & $\left(a,v\right)$ & $\left(v-a,v-a\right)$  & \text{\bf Discriminant} &  \text{\bf Interpretation} \\
\hline
-2 & 0  &   0 & -4 &           \text{$\bP^1$-bundle over S} \\
\hline
-2 & 1  & -2 & -5&    \text{$\bP^2$} \\
\hline
 0 & 1  & 0  & -1&  \text{$\bP^1$-bundle over S} \\
\hline
\end{tabular}
\end{center}

The only inclusion of lattices takes the first case to the third:
$$
\begin{array}{c|cc}
	& v & a \\
\hline 
v & 2 & 1 \\
a & 1 & 0 
\end{array}
\supsetneq
\begin{array}{c|cc}
	& v & a' \\
\hline 
v & 2 & 2 \\
a' & 2 & 0 
\end{array}
$$
This is induced by $a'=2a$ which does not correspond to a basic decomposition.
We know the exceptional locus of the Hilbert-Chow contraction is irreducible,
so Question~\ref{ques:organize} has a positive answer in this case.


\subsection{Dimension six}
\label{subsect:fourcase}

We take $\left(v,v\right)=4$.
The case of Lagrangian $\bP^3$'s, where $\left(a,a\right)=-2$ and $\left(a,v\right)=2$,
was examined in \cite{HHT3}.

\begin{center}
\begin{tabular}{|c|c|c|c|l|}
\hline
$\left(a,a\right)$ & $\left(a,v\right)$  & $\left(v-a,v-a\right)$ & \text{\bf Discriminant}   & \text{\bf Interpretation} \\
\hline
 -2 & 0  & 2& -8&                                      \text{$\bP^1$-bundle over $S^{[2]}$} \\
\hline
 -2 & 1 & 0 & -9 &  \text{$\bP^2$-bundle over $S$} \\
\hline
 -2 & 2 & -2 & -12&   \text{$\bP^3$} \\
\hline
 0 & 1 & 2 & -1  & \text{$\bP^1$-bundle over $S\times S$} \\
\hline
 0 & 2 & 0 & -4   &\text{$\bP^1$-bundle
over $S\times S'$,} \\
& & & &  \text{$S$ and $S'$ are isogenous}\\
\hline
\end{tabular}
\end{center}
Again, the only possible inclusion involves the Hilbert-Chow case, which
has irreducible exceptional locus.

The last entry was omitted in \cite[Table H3]{HT4} but was included
in the general conjecture proposed in that paper.
We sketch the geometry in this case:
Suppose 
$$
X=M_v(T),\quad v=(r,Nh,s), \quad N\neq 0
$$ 
for 
some K3 surface $(T,h)$; we assume that $a=(r',N'h,s')$.
Express
$$H^2(X,\bZ)=v^{\perp} \subset \wLambda\simeq U \oplus H^2(T,\bZ)$$
so there is a saturated embedding of the primitive cohomology
$$H^2(T,\bZ)_{\circ}=h^{\perp} \hookrightarrow H^2(X,\bZ).$$
The factors of the center of $X\ra X'$ are $S=M_a(T)$ and $S'=M_{v-a}(T)$,
which have cohomology groups
$$H^2(M_a(T),\bZ)= a^{\perp}/\bZ a , \quad
H^2(M_{v-a}(T),\bZ)= (v-a)^{\perp}/\bZ (v-a).$$
We also have embeddings
$$H^2(T,\bZ)_{\circ} \hookrightarrow H^2(S,\bZ), \quad
H^2(T,\bZ)_{\circ} \hookrightarrow H^2(S',\bZ),$$
that fail to be saturated in some cases.

\subsection{Dimension eight} Here we have $\left(v,v\right)=6$:

\begin{center}
\begin{tabular}{|c|c|c|c|l|}
\hline
$\left(a,a\right)$ & $\left(a,v\right)$  & $\left(v-a,v-a\right)$ & \text{\bf Discriminant} &  \text{\bf Interpretation} \\
\hline
 -2 & 0  & 4& -12 &                                 \text{$\bP^1$-bundle over $S^{[3]}$} \\
\hline
 -2 & 1  & 2 & -13 &   \text{$\bP^2$-bundle over $S^{[2]}$} \\
\hline
 -2 & 2  & 0& -16 &  \text{$\bP^3$-bundle over $S$} \\
\hline
 -2 & 3  & -2  & -21 & \text{$\bP^4$} \\
\hline
 0 & 1   &  4& -1 &\text{$\bP^1$-bundle over $S\times S^{[2]}$} \\
\hline
 0 & 2   & 2& -4 & \text{$\bP^1$-bundle
over  $S' \times S^{[2]}$} \\
& & & &  \text{$S,S'$ are isogenous}\\
\hline
 0 & 3   & 0& -9 & \text{$\bP^2$-bundle over $S \times S'$} \\
& & & &   \text{$S,S'$ are isogenous}\\

\hline
\end{tabular}
\end{center}
The last entry was also omitted in \cite[Table H4]{HT4} but included
in the general conjecture. The geometry is similar to example 
in Section~\ref{subsect:fourcase}.

Here we do have an inclusion of lattices not involving the Hilbert-Chow example.
Write
$$
\cH=\begin{array}{c|cc}  
		 & v & a \\
		\hline
	v	& 6 & 2 \\
	a 	& 2 & 0 \end{array}
\supsetneq
\cH'=\begin{array}{c|cc}  
		 & v & a' \\
		\hline
	v	& 6 & 2 \\
	a' 	& 2 & -2 \end{array}
$$
with $a'=v-2a$ and $v-a'=2a$.
However, this is not a basic decomposition so it
does not arise from an additional component of the exceptional 
locus.  Note that in the decomposition
$$v=a'+(v-a')=a'+2a$$
we have an isotropic vector with multiplicity two.

\subsection{Dimension ten} In this case $\left(v,v\right)=8$.

\begin{center}
\begin{tabular}{|c|c|c|c|l|}
\hline
$\left(a,a\right)$ & $\left(a,v\right)$  & $\left(v-a,v-a \right)$ & \text{\bf Discriminant}  &  \text{\bf Interpretation} \\
\hline
 -2 & 0  &   6 & -16 &                                   \text{$\bP^1$-bundle over $S^{[4]}$} \\
\hline
 -2 & 1 & 4 & -17 &  \text{$\bP^2$-bundle over $S^{[3]}$} \\
\hline
 -2 & 2  & 2 & -20 &  \text{$\bP^3$-bundle over $S^{[2]}$} \\
\hline
 -2 & 3  & 0& -25&  \text{$\bP^4$-bundle over $S$} \\
\hline
 -2 & 4  & -2 & -32& \text{$\bP^5$} \\
\hline
 0 & 1 & 6 & -1  &\text{$\bP^1$-bundle over $S\times S^{[3]}$} \\
\hline
 0 & 2 & 4 & -4 &   \text{$\bP^1$-bundle
over  $S' \times S^{[3]}$} \\
& & & &  \text{$S,S'$ are isogenous}\\
\hline
 0 & 3 & 2 & -9  &\text{$\bP^2$-bundle over  $S' \times S^{[2]}$} \\
& & & & \text{$S,S'$ are isogenous} \\
\hline
0 & 4 & 0 & -16 & \text{$\bP^3$-bundle over  $S \times S'$} \\
& & & & \text{$S,S'$ are isogeneous} \\
\hline
\end{tabular}
\end{center}
The sublattices
$$
\begin{array}{c|cc} 
	& v & a \\
\hline
v	& 8 & 1 \\
a       & 1 & 0
\end{array}
\supsetneq
\begin{array}{c|cc} 
	& v & Na \\
\hline
v	& 8 & N \\
Na       & N & 0
\end{array}
$$
explain all embeddings among the lattices in the table above,
preserving $v$. None of these arise from basic decompositions.

\subsection{Higher dimensional data}
\label{subsect:twelve}
For later applications, we enumerate possible discriminants of the lattice $\left<v,a\right>$.
Note that this lattice need not be saturated in the Mukai lattice, so each row below
may correspond to multiple cases.
When $\left(v,v\right)=10$ we have
\begin{center}
\begin{tabular}{|c|c|l|}
\hline
$\left(a,a\right)$ & $\left(a,v\right)$  & \text{\bf Discriminant} \\
\hline
$-2$   & $b=0,1,2,3,4,5$  & $-20-b^2$ \\
$0$ & $b=0,1,2,3,4,5$ & $-b^2$ \\
$2$ & $b=5$           & $20-b^2$ \\
\hline
\end{tabular}
\end{center}
The case $\left(v,v\right)=12$ yields
\begin{center}
\begin{tabular}{|c|c|l|}
\hline
$\left(a,a\right)$ & $\left(a,v\right)$  & \text{\bf Discriminant} \\
\hline
$-2$   & $b=0,1,2,3,4,5,6$  & $-24-b^2$ \\
$0$ & $b=0,1,2,3,4,5,6$ & $-b^2$ \\
$2$ & $b=5,6$           & $24-b^2$ \\
\hline
\end{tabular}
\end{center}

\subsection{Characterizing Lagrangian $\bP^{n}$'s}
Smoothly embedded rational curves in a K3 surface
$$\ell:=\bP^1 \subset S$$
are characterized as $(-2)$-curves $\left(\ell,\ell\right)=-2$. 
Suppose that $X$ is deformation
equivalent to $S^{[n]}$ and we have a smoothly embedded
$$\bP^n \subset X$$
with $\ell \subset \bP^n$ a line. For $n=2,3$ we showed in \cite{HT2,HHT3}
that these are unique up to monodromy and satisfy
$$\left( \ell, \ell \right)=-\frac{n+3}{2}.$$
For $n=4$ Bakker and Jorza \cite{BJ} computed 
$$\left(\ell,\ell\right)=-\frac{7}{2}.$$
Furthermore, Bakker 
\cite[Cor.~23]{Bakker} has offered sufficient conditions to guarantee
that Lagrangian planes form a single monodromy orbit.  
Previously \cite[Thesis~1.1]{HT4}, we suggested that the intersection
theoretic properties of these classes should govern the cone of
effective curves. Markman, Bayer and Macr\`i offered counterexamples
to our original formulation \cite[Conj.~1.2]{HT4} in \cite[\S 10]{BM1}.

Our purpose here is to illustrate that there may be multiple orbits
of Lagrangian projective spaces under the monodromy action. For each $n$, the
lattice
$$\cG=\begin{array}{c|cc}
	& v & a \\
\hline
v    & 2(n-1) & n-1 \\
a    & n-1 & -2
\end{array}
$$
gives rise to a Lagrangian projective space with 
$$R=[\ell]=\pm \theta^{\vee}(a).$$
We expect a second orbit in cases when there exists an embedding
$\cG \hookrightarrow \cH$ as a finite index sublattice
(cf.~\cite[\S 14]{BM2} and Question~\ref{ques:organize}).
For these to exist, the discriminant 
$$\operatorname{disc}(\cG)=-(n-1)(n+3)$$
should be divisible by a square.
It is divisible by $9$ when $n\equiv 6\pmod{9}$;
we consider when $n=15$.

Let $\cH$ denote the lattice
$$\cH = \begin{array}{c|cc}
	     & v & a \\
	\hline  
	v & 28 & 14 \\
 	a & 14 & 6 
\end{array}
$$
which has discriminant $-28$.  
The lattice 
$$\cG = \begin{array}{c|cc}
	& v & a' \\
\hline
v &   28 & 14 \\
a' &   14 & -2 
\end{array}
$$
can be realized as an index three sublattice of $\cH$
via $a'=3a-v$. Thus we obtain {\em two} basic decompositions of $v$.
The associated contraction $X\ra X'$ thus should have reducible exceptional locus, with one component
isomorphic to $\bP^{15}$ and the other of codimension $\left(a,v-a\right)-1=7$.

\subsection{Another example with interesting exceptional locus}
Consider the lattice
$$\cH= \begin{array}{r|cc}
	& v & a \\
\hline
v & 10 & 5 \\
a & 5 &  2
\end{array}
$$
associated with a $12$-dimensional holomorphic symplectic manifold $X$.

Note that $v^{\perp}$ is generated by a spherical class $s_1=v-2a$,
which we take to be effective.
Thus we have the decomposition
$$v=2a + (v-2a)$$
which yields a codimension-one stratum in $X$ isomorphic to a $\bP^1$-bundle
over $M_{2a}$.
The formalism of Bayer-Macr\`i \cite[Lem.~7.5]{BM2} implies there is one exceptional divisor 
arising from the extremal ray associated with $\cH$.

The lattice $\cH$ represents $(-2)$ infinitely many times.
Here are the vectors $b$ with $\left(b,b\right)=\pm 2$
and $\left(v,b\right) \ge 0$:
$$\begin{array}{r|cc}
b   & \left(b,b\right) & \left(v,b\right) \\
\hline
2a-v & -2 & 0 \\
s_3=3a-v & -2 & 5 \\
7a-2v & -2 & 15 \\
\vdots & \vdots & \vdots \\
4a-v & 2 & 10 \\
a   & 2 & 5 \\
v-a & 2 & 5 \\
3v-4a & 2 & 10 \\
\vdots & \vdots & \vdots \\
5v-7a & -2 & 15 \\
s_2=2v-3a & -2 & 5 \\
s_1=v-2a & -2 & 0 
\end{array}
$$
While we have 
$$\begin{array}{r|cc}
	& s_2 & s_3 \\
\hline
s_2 & -2 & 7 \\
s_3 & 7 & -2 
\end{array}.
$$
the decomposition
$$v=s_2+s_3$$
is not basic as $\left(s_1,s_2\right)<0$.
We do not expect this decomposition to correspond to a Lagrangian $\bP^6$.

\begin{exam}
Here is a concrete example: Let $T$ denote a K3 surface with
$\operatorname{Pic}(T)=\bZ f$ with $f^2=22$. Let 
$v=(1,f,6)$ and $a=(-1,0,1)$. Elements of $M_v(T)$ generically
take the form $\cI_{\Sigma}(f)$ where $\Sigma \subset T$
is of length six. The distinguished spherical class is
$v-2a=(3,f,4),$ 
which arises in the uniform construction of $T$
in \cite[\S 3]{MukaiGenus13}.
\end{exam}

\section{Orbits and extremal rays}
\label{sect:OER}

We fix a primitive vector $v \in \wLambda$ 
such that $v^{\perp}=H^2(X,\bZ)$, as in Section~\ref{sect:generalities}.
Write 
\begin{equation} \label{eqn:defvarrho}
\left(v,v\right)a-\left(a,v\right)v=M\varrho
\end{equation}
where $M>0$ and $\varrho \in v^{\perp}$ is primitive.
The {\em divisibility} $\dv(\varrho)$ is defined as the positive integer
such that 
$$
\left(\varrho,H^2(X,\bZ)\right)=\dv(\varrho)\bZ,
$$
so that
$R=\varrho/\dv(\varrho)$ represents (via duality) a class in $H_2(X,\bZ)$ and
an element of the discriminant group $d(H^2(X,\bZ))=H_2(X,\bZ)/H^2(X,\bZ)$.  

Note that $a$ projects to a negative class in $v^{\perp}$ if and only if
$$\left(a,a\right) \left(v,v\right) < \left(a,v\right)^2,$$
i.e., the lattice $\left<a,v\right>$ has signature $(1,1)$.
The autoduality of the positive cone and the fact that nef divisors 
have non-negative Beauville-Bogomolov
squares imply that the Mori cone contains the positive cone.  Thus we may
restrict our attention to $\varrho$ with $\left(\varrho,\varrho\right)<0$.

We exhibit representatives of these orbits in the special case where $X=S^{[n]}$;
we use Example~\ref{exam:hilb2} and write 
$\delta^{\vee}=\frac{1}{2n-2}\delta$ so that
$$d(H^2(S^{[n]},\bZ))=(\bZ/2(n-1)\bZ)\cdot \delta^{\vee}.$$
Our objective is to write down {\em explicit} examples where they arise from extremal
rational curves:
\begin{theo} \label{theo:existray}
Retain the notation introduced above and assume that $R^2<0$.  
Then there exists a K3 surface
$S$ with $\Pic(S) \simeq \bZ f$ that admits an extremal rational curve
$\bP^1 \subset S^{[n]}$ such that $\bR_{\geq 0}[\bP^1]$ is equivalent
to $\bR_{\geq 0}R$ under the action of the monodromy group.
\end{theo}
In particular, the cone of effective curves of $S^{[n]}$ is generated by 
$\delta^{\vee}$ and $[\bP^1]$.

We will develop several lemmas to prove this theorem.
A direct computation (cf.~\cite[Prop.~12.6]{BM2}) gives:
\begin{lemm} \label{lemm:bound}
Retain the notation introduced above and assume that 
$$
\left(v,v\right)=2(n-1), \quad
|\left(a,v\right)|\leq v^2/2, \,\,\text{  and } \, a^2 \ge -2.
$$  
Then we have
$$\left(\rho,\rho\right) > -2(n-1)^2(n+3), \quad
\left(R,R\right)>-(n+3)/2.$$
\end{lemm}

Markman \cite[Lemma~9.2]{MarkSurv}
shows that the image $G_n$ of the monodromy representation consists of the orientation-preserving
automorphisms of the lattice $H^2(X,\bZ)$ acting via $\pm 1$ on $d(H^2(X,\bZ))$.
In particular, $d(H^2(X,\bZ))$ has a distinguished generator 
$\pm \delta^{\vee}$, determined
up to sign.  
We consider orbits of primitive vectors $\varrho \in H^2(X,\bZ)=v^{\perp}$ under
the action of automorphisms of $H^2(X,\bZ)$ acting trivially on $d(H^2(X,\bZ))$.  
A classical result of Eichler \cite{Eichler} (see also \cite[Lemma 3.5]{GHS})
shows that there is a unique orbit of primitive elements $\varrho'\in H^2(X,\bZ)$ such that
\begin{equation} \label{eqn:varrho}
\left(\varrho',\varrho'\right)=\left(\varrho,\varrho\right), \quad
\varrho'/\dv(\varrho')=\varrho/\dv(\varrho) \in d(H^2(X,\bZ)).
\end{equation}
The same holds true even if we restrict to the subgroup preserving orientations.
Let $G_n^+ \subset G_n$ denotes the orientation preserving elements;
for this group, the second part of (\ref{eqn:varrho}) may be relaxed to
$$\varrho'/\dv(\varrho')=\pm \varrho/\dv(\varrho) \in d(H^2(X,\bZ)).$$

\begin{lemm}  \label{lemm:orbitlist}
Each $G^+_n$-orbit of primitive vectors in $H^2(X,\bZ)$ has a representative
of the form
\begin{equation} \label{eqn:standard}
\varrho=sf-t\delta, \quad \gcd(s,t)=1, \ s,t>0, \ s|2(n-1),
\end{equation}
where $f \in H^2(S,\bZ)$ is primitive with $f^2=2d>0$.
Here $\dv(\varrho)=s$, $R=\varrho/s$, and $[R]=-2t(n-1)/s \in d(H^2(X,\bZ))$.  
\end{lemm}
This is quite standard---see the first paragraph of the proof of
\cite[Prop.~3.6]{GHS} for the argument via Eichler's criterion.

\begin{lemm} \label{lemm:produce}
Fix a constant $C$ and the orbit of a primitive vector $\varrho \in H^2(S^{[n]},\bZ)$ with 
$C \le  \varrho^2 < 0$ and $\dv(\varrho)=s$.  
Then there exists an even integer $2d>0$ and a 
representation \eqref{eqn:standard}
such that for every 
$$\varrho_0=\sigma f -\tau \delta, \ \sigma,\tau>0$$
with $C\le \varrho_0^2<0$ we have
$t/s > \tau/\sigma$.
\end{lemm}
\begin{proof}
First, let $\mu<s$ be a positive integer such that $t+\mu$ is divisible by $s$.
If we express $t/s$ as a continued fraction
$$t/s=[a_0,a_1,\ldots,a_r]$$
then $[a_1,\ldots,a_r]$ depends only on $\mu/s$ and $a_0=\lfloor t/s \rfloor$.
We regard $a_1,\ldots,a_r$ as fixed and $a_0$ as varying.

If the representation \eqref{eqn:standard} is to hold we must have
$$\varrho^2= 2ds^2-2(n-1)t^2$$
which implies
\begin{eqnarray*}
d&=&(n-1)\left(\frac{t}{s}\right)^2 + \frac{\varrho^2}{2s^2} \\
 &=& \frac{2(n-1)t^2+\varrho^2}{2s^2}.
\end{eqnarray*}
If the fraction is an integer for some $t$ it is an 
integer for an arithmetic sequence of $t$'s.  
Thus there are solutions for $t\gg 0$, and we may assume $d$ is large.  

Now suppose that $\tau_j/\sigma_j=[a_0,\ldots,a_j]$ for some $j<r$.  We estimate
$$
\frac{d}{n-1}-\left(\frac{\tau_j}{\sigma_j}\right)^2 
=\left(\frac{t}{s}\right)^2+\frac{\varrho^2}{2(n-1)s^2}-\left(\frac{\tau_j}{\sigma_j}\right)^2 $$
using the continued fraction expressions.  Substituting yields
$$
(a_0+\frac{1}{[a_1,\ldots,a_r]})^2 +\frac{\varrho^2}{2(n-1)s^2}-(a_0+\frac{1}{[a_1,\ldots,a_j]})^2$$
and cancelling the $a_0^2$ terms gives
$$ 2a_0 (\frac{1}{[a_1,\ldots,a_r]}-\frac{1}{[a_1,\ldots,a_j]})+
(\frac{1}{[a_1,\ldots,a_r]}-\frac{1}{[a_1,\ldots,a_j]})^2+\frac{\varrho^2}{2(n-1)s^2}.
$$
This can be made arbitrarily large in absolute value if $a_0\gg 0$.  
Therefore, for $j<r$ we conclude
$$2d \sigma_j^2 -2(n-1)\tau_j^2 \not \in [C,0).$$

Suppose we have $\sigma$ and $\tau$ as specified above in the assumption of the Lemma.
It follows that 
$$\frac{C}{2(n-1)\sigma^2} < \frac{d}{n-1}-\left(\frac{\tau}{\sigma}\right)^2 < 0;$$
dividing both sides by $\sqrt{d/(n-1)}+\frac{\tau}{\sigma}$, which we may assume is 
larger than $\frac{|C|}{(n-1)}$, we obtain
$$\frac{1}{2\sigma^2}>|\sqrt{d/(n-1)} -\frac{\tau}{\sigma}|.$$
It follows (see \cite[Thm.~184]{HW}, for example) that $\tau/\sigma$ is necessarily
a continued fraction approximation for $\sqrt{d/(n-1)}$, say, $\tau_{r'}/\sigma_{r'}$.  
Given a representation \eqref{eqn:standard} we may assume
that $t/s$ is a continued fraction approximation as well.

Let $\tau_j/\sigma_j=[a_0,\ldots,a_j]$ denote the sequence of continued fraction approximations of
$\sqrt{d/(n-1)}$, starting from
$$\tau_0=\lfloor \sqrt{d/(n-1)}\rfloor, \sigma_0=1.$$ 
Note that 
$$\tau_{2w-2}/\sigma_{2w-2}<\tau_{2w}/\sigma_{2w}<\sqrt{d/(n-1)}<\tau_{2w+1}/\sigma_{2w+1}<\tau_{2w-1}/\sigma_{2w-1}$$
for each $w\in \bN$, thus
$$2d\sigma_j^2-2(n-1)\tau_j^2<0$$
precisely when $j$ is odd.  
Our estimate above shows that $r'>r$ whence
$$\tau/\sigma=\tau_{r'}/\sigma_{r'}<\tau_r/\sigma_r=t/s,$$
which is what we seek to prove.  \end{proof}

\begin{proof}We complete the proof of Theorem~\ref{theo:existray}.  

Lemma~\ref{lemm:bound} shows that each $\varrho_0 \in H^2(S^{[n]},\bZ)$ associated with
a negative extremal rays satisfies 
$$ C=-2(n+3)(n-1)^2 \le \varrho_0^2 <0.$$
Lemma~\ref{lemm:orbitlist} allows us to assume $\varrho_0$ is equivalent under the monodromy
action to one of the lattice vectors
satisfying the hypotheses of Lemma~\ref{lemm:produce}.

Take $S$ to be a K3 surface with $\Pic(S)=\bZ f$ and $f^2=2d$; thus we have
$\Pic(S^{[n]})=\bZ f \oplus \bZ \delta$.   The cone of effective
curves of $S^{[n]}$ has two generators, one necessarily $\delta^{\vee}$.
We choose $d$ via Lemma~\ref{lemm:produce}.  
We know from \cite[Thm.~12.2]{BM2} that the generator of $\bQ_{\ge 0}\varrho
\cap H_2(S^{[n]},\bZ)$, is effective with some multiple generated
by a rational curve $\bP^1 \subset S^{[n]}$.  However, Lemma~\ref{lemm:produce} ensures that
all the other $\varrho_0 \in \Pic(S^{[n]})$ satisfying 
$$C \le \varrho_0^2 <0, \quad \left(\varrho_0,f\right)>0,$$
are contained in the cone spanned by $\delta^{\vee}$ and $\varrho$.  
Thus our rational curve is necessarily extremal.  
\end{proof}

\section{Automorphisms on Hilbert schemes not coming from K3 surfaces}
\label{sect:auto}

Oguiso-Sarti asked whether $S^{[n]}, n\ge 3$ can admit automorphisms not 
arising from automorphisms of $S$.  
Beauville give examples for $n=2$, e.g., the secant line involution for 
generic quartic surfaces; recently, a systematic analysis has been offered in \cite{BCNWS}.
A related question of Oguiso is to exhibit automorphisms of $S^{[n]}$ to arising
from automorphisms of any K3 surface $T$ with $T^{[n]}\simeq S^{[n]}$
(see Question 6.7 in his ICM talk \cite{oguiso}).

\begin{prop}
There exists a polarized K3 surface $(S,h)$ such that $S^{[3]}$
admits an automorphism $\alpha$ not arising from $S$.  Moreover, there exists 
no K3 surface $T$ with $T^{[3]}\simeq S^{[3]}$ explaining $\alpha$.
\end{prop}
For simplicity we will 
restrict to those with $\Pic(S)=\bZ h$ with $h^2=d$. 
We have
$$ 
\Pic(S^{[3]})=\bZ h \oplus \bZ\delta,
$$ 
with $2\delta$
the class of the non-reduced subschemes. Recall that 
$$
(h,h)=d, \quad        (h,\delta)=0,  \quad         (\delta,\delta)=-4.
$$

\begin{lemm}
Suppose there exists an element $g\in \Pic(S^{[3]})$ such that
$$
(g,g)=2 \quad \text{and}\quad g.R>0
$$
for each generator $R$ of the cone of effective curves
of $S^{[3]}$.  
Then $S^{[3]}$ 
admits an involution associated with reflection
in $g$:
$$
D \mapsto  -D + (D,g) g.
$$
\end{lemm}

\begin{proof}
This follows from the Torelli Theorem, as the reflection is a monodromy operator in the sense of Markman.
\end{proof}

\begin{exam}
Let $d=6$.  Given three points on a degree six K3 surface, the plane they
span meets the K3 surface three additional points, yielding an involution
$S^{[3]} \dashrightarrow S^{[3]}$.

However, this breaks down along triples of collinear points, which are
generally parametrized by maximal isotropic subspaces of the (unique smooth)
quadric hypersurface containing X.  These are parametrized by a $\bP^3 \subset S^{[3]}$.
Here we have $g=h-\delta$ and the offending $R$ 
is Poincar\'e dual to a multiple of $2h-3\delta$.  
The class of the line in $\bP^3$ is $h-(3/2)\delta$,
interpreting $H_2(S^{[3]},\bZ)$ as a finite extension of $H^2(S^{[3]},\bZ)$.

Returning to arbitrary $d$, we apply the ampleness criterion to find the extremal curves.  
One is proportional to $\delta$.  The second generator is given by
$R=ah-b\delta$, with $(a,b)$ non-negative relatively prime integers 
satisfying one of the following:
\begin{enumerate}
\item   $da^2-4b^2=-2$
\item   $da^2-4b^2=-4$, with $a$ divisible by 4
\item   $da^2-4b^2=-4$, with $a$ divisible by 2 but not 4
\item   $da^2-4b^2=-12$, with $a$ divisible by 2 but not 4
\item   $da^2-4b^2=-36$, with $a$ divisible by 4
\end{enumerate}
The smallest example is $d=114$ and $g=3h-16\delta$.  To check
the ampleness criterion, the first step is to write down all the 
$(a,b)$ where $114a^2-4b^2$ is
`small' using the continued fraction expansion
$$
        \sqrt{114}/2=[5;2,1,20,1,2,10],
$$
which gives the following
$$
\begin{array}{r|r|r}
        a    &   b   &    114a^2-4b^2 \\
\hline
        1    &   5   &       14 \\
        2    &   11  &     -28 \\
        3    &   16  &     2 \\
        62   &    331 &     -28 \\ 
        65   &    347 &     14 \\
        127  &    678 &     -30 \\
        192  &    1025 &    -4 
\end{array}
$$
The class $R=192h-1025\delta$ is the second extremal generator; note it
satisfies $R.g=64>0$ which means that $g$ is ample on $S^{[3]}$.

Now $-36$ is not `small' for $114a^2-4b^2$ so we need to analyze this case
separately.  However, the equation
$$
        114a^2-4b^2=-36
$$
only has solutions when $a$ and $b$ are both divisible by three.
\end{exam}

\section{Ambiguity in the ample cone}
\label{sect:ambiguity}
The following addresses a question raised by Huybrechts:
\begin{theo}
There exist polarized manifolds of K3 type $(X,g)$ and $(Y,h)$ admitting
an isomorphism of Hodge structures
$$
\phi:H^2(X,\bZ) \ra H^2(Y,\bZ), \quad \phi(g)=h
$$
not preserving ample cones.
\end{theo}

This contradicts our speculation that the Hodge structure determines the ample cone of a polarized holomorphic symplectic manifold;
we also need to keep track of the Markman extension data.

We first explain the idea: Let $\Lambda_n$ denote the lattice isomorphic to $H^2(X,\bZ)$
where $X$ is deformation equivalent to $S^{[n]}$ where $S$ is a K3 surface.
Given an isomorphism $X\simeq S^{[n]}$ we have a natural embedding $\Lambda_n \hookrightarrow \wLambda$.
Let $d(\Lambda_n)$ denote the discriminant group with the associated $(\bQ/2\bZ)$-valued
quadratic form. There is a natural homomorphism
$$\Aut(\Lambda_n) \ra \Aut(d(\Lambda_n))$$
which is surjective by Nikulin's theory of lattices.
The automorphisms of $\Lambda_n$ extending to automorphisms of $\wLambda$ are those
acting via $\pm 1$ on $d(\Lambda_n)$ \cite[\S 9]{MarkSurv}.

We choose $n$ such that $\Aut(d(\Lambda_n)) \supsetneq \{ \pm 1\}$,
exhibit an $\alpha \in \Aut(\Lambda_n)$ not mapping to $\pm 1$,
and show that $\alpha$ fails to preserve the ample cone by
verifying that its dual
$$\alpha^*:\Lambda^*_n \ra \Lambda^*_n$$
fails to preserve the extremal rays identified by Bayer-Macr\`i.

We start by fixing notation: Consider
$$\Lambda_n \subset \wLambda \simeq U \oplus H^2(S,\bZ) \simeq U^4 \oplus (-E_8)^2$$
realized as the orthogonal complement of a vector $v \in U\simeq H^2(S,\bZ)^{\perp}$.
Let $e_1,f_1$ denote a basis for this $U$ satisfying
$$\left(e_1,e_1\right)=\left(f_1,f_1\right)=0, \left(e_1,f_1\right)=1;$$
let $e_2,f_2$ denote a basis for one of the hyperbolic summands
$U\subset H^2(S,\bZ)$. We may assume $v=e_1+nf_1$ and write 
$\delta=e_1-nf_1$.
Since $\Lambda_n^*\simeq H_2(X,\bZ)$ the classification of extremal
rays is expressed via monodromy orbits of vectors $R\in H_2(X,\bZ)$.
The pre-image of $\bZ R$ in $\wLambda$ is a rank-two lattice
$$ \cH \subset \wLambda, \quad   v,a \in \wLambda,$$
where $a$ is as described in Theorem~\ref{theo:main}.

The first step is to give an $n$ such that the group
$$(\bZ/2(n-1)\bZ)^*$$
admits an element $\bar{\alpha}\neq \pm 1$ such that
$$\bar{\alpha}^2 \equiv 1 \pmod{4(n-1)}.$$
We choose $n=7$ and $\bar{\alpha}=5$.

Next, we exhibit an $\alpha \in \Aut(\Lambda_7)$ mapping to 
$\bar{\alpha}$.
These exist by Nikulin's general theory, but we
offer a concrete example of such an automorphism.
Then we may take
$$\alpha(\delta)=5\delta+12(e_2+f_2), \ \alpha(e_2)=\delta+2e_2+3f_2,\ \alpha(f_2)=\delta+3e_2+2f_2$$
and acting as the identity on the other summands. 

The third step is to find an extremal ray that fails to be sent to an
extremal ray under $\alpha^*$. We are free to pick any representative
in the orbit under the monodromy. Consider then the
lattice
$$\cH_1:=\begin{array}{c|cc}
		& v & a \\
\hline
v & 12 & 5 \\
a & 5 & -2 
\end{array}
$$
with 
$$a=5f_1+e_2-f_2.$$
Consider the element 
$$a'=v-5a \in \Lambda_7;$$
the relevant ray $R$ is a generator of
$$\bQ a' \cap \Lambda_7^* \subset \Lambda_7 \otimes \bQ.$$
Explicitly
$$a'=5\delta - 12 e_2 + 12f_2$$
and 
$$\alpha(a')=25\delta + 72 e_2+48 f_2=25e_1-150f_1 + 72e_2+48f_2.$$
Let $\cH_2$ denote the saturated lattice containing $\alpha(a')$ and $v$.
Note that 
$\alpha(a')-v$ is divisible by $12$; write
$$b=\frac{\alpha(a')-v}{12}=2e_1-13f_1+6e_2+4f_2.$$
In particular, $\left<v,b\right> \subset \wLambda$ is saturated.
Thus we find:
$$\cH_2= \begin{array}{c|cc}
		& v & b \\
\hline 
v & 12 & -1 \\
b & -1 & -4
\end{array}
$$

We put $\cH_1$ and $\cH_2$ in reduced form: 
$$\cH_1 \simeq \left( \begin{matrix} 0 & 7 \\ 7 & -2 \end{matrix} \right), \quad
\cH_2 \simeq \left( \begin{matrix} 0 & 7 \\ 7 & -4 \end{matrix} \right)$$
which are {\em inequivalent} lattices of discriminant $-49$.
We refer the reader to the Section~\ref{subsect:twelve}: there is a 
{\em unique} lattice that appears of discriminant $-49$, i.e.,
the one associated with $\cH_1$. Thus $\cH_2$ is not associated with an
extremal ray $R'$. 

To recapitulate: Suppose we started with an $X$ such that the 
vector $a$ yields an extremal ray $R$. We apply the automorphism 
$\alpha$ to $H^2(X,\bZ)$ to get a new Hodge structure, equipped
with an embedding into $\wLambda$;
surjectivity of Torelli \cite{HuyBasic} guarantees the
existence of another hyperk\"ahler manifold $Y$ with this
Hodge structure and a compatible embedding $H^2(Y,\bZ) \subset \wLambda$.
However, the class $R'\in H^2(Y,\bZ)$ corresponding to $R$ is
{\em not} in the monodromy orbit of any extremal ray.  

\

To exhibit a concrete projective example of this type, we could carry out an
analysis along the lines of Theorem~\ref{theo:existray} in
Section~\ref{sect:OER}. There we showed that each monodromy orbit of 
extremal rays $R$ arises as from an extremal rational curve 
$$\bP^1 \subset S^{[n]} \simeq X$$
where $(S,A)$ is a polarized K3 surface, perhaps of very large degree. The approach
was to show that the {\em only} vectors in $\Pic(X)$ with `small' norm 
are $\delta$ and $\varrho$, a positive integer multiple of $R$.

What happens when we apply the construction above to such an $X\simeq S^{[n]}$?
The isomorphism 
$$\alpha: H^2(X,\bZ)\stackrel{\sim}{\ra}  H^2(Y,\bZ)$$
implies
$\Pic(Y)\simeq \Pic(X)$ as lattices, so their
small vectors coincide. Furthermore, we may choose 
$Y\simeq T^{[n]}$
where $(T,B)$ is a polarized K3 surface isogenous to $(S,A)$, i.e.,
we have isomorphisms of polarized integral Hodge structures
$$H^2(S,\bZ) \supset A^{\perp} \simeq B^{\perp} \subset H^2(T,\bZ).$$
Moreover, we may assume that $\delta_X$ is taken to $\delta_Y$,
i.e., the extremal curve class $\delta^{\vee}_{S^{[n]}}$ maps to $\delta^{\vee}_{T^{[n]}}$.
Consequently, there exists an ample divisor on $S^{[n]}$---for instance,
$g:=N A -\delta_{S^{[n]}}$ for $N\gg 0$---that goes to an ample divisor $h=\alpha(g)$ on $T^{[n]}$.

Let $\varrho \in \Pic(X)$ denote the class arising as a positive multiple
of the extremal ray; note that $\varrho=\pm a'$ in the notation above.
Now $\alpha(\varrho)$ does not correspond to an effective class, so the second
extremal ray on $Y$ corresponds to a subsequent vector of `small' norm,
i.e., 
$$\alpha^*(\text{cone of effective curves on $X$})\subsetneq
\text{cone of effective curves on $Y$}.$$

\begin{rema}
Markman has independently obtained an example along
these lines; it is also
of K3 type,
deformation equivalent to $S^{[7]}$.
\end{rema}

\begin{rema}
Explicit descriptions of cones of divisors
on generalized Kummer manifolds have been
found by Yoshioka \cite{Yosh}. 
Qualitative descriptions of these cones, with applications
to the Kawamata-Morrison conjecture, have been established by
Markman-Yoshioka \cite{MarkYosh} and Amerik-Verbitsky 
\cite{AmerVerb}.
\end{rema}

\bibliographystyle{alpha}
\bibliography{BMSurvey}

\begin{thebibliography}{BCNWS14}

\bibitem[AV14]{AmerVerb}
Ekaterina Amerik and Misha Verbitsky.
\newblock Morrison-{K}awamata cone conjecture for hyperkahler manifolds, 2014.
\newblock arXiv:1408.3892.

\bibitem[Bak13]{Bakker}
Benjamin Bakker.
\newblock A classification of extremal {L}angrangian planes in holomorphic
  symplectic varieties, 2013.
\newblock arXiv:1310.6341.

\bibitem[BCNWS14]{BCNWS}
Samuel Boissi\`ere, Andrea Cattaneo, Marc Nieper-Wisskirchen, and Alessandra
  Sarti.
\newblock The automorphism group of the {H}ilbert scheme of two points on a
  generic projective {$K3$} surface, 2014.
\newblock arXiv:1410.8387.

\bibitem[BHT13]{BHT}
Arend Bayer, Brendan Hassett, and Yuri Tschinkel.
\newblock Mori cones of holomorphic symplectic varieties of {$K3$} type.
\newblock {\em Annales scientifiques de l'\'Ecole normale sup\'erieure}, to
  appear, 2013.
\newblock arXiv:1307:2291.

\bibitem[BJ14]{BJ}
Benjamin Bakker and Andrei Jorza.
\newblock Lagrangian 4-planes in holomorphic symplectic varieties of
  {$\text{K3}^{[4]}$}-type.
\newblock {\em Cent. Eur. J. Math.}, 12(7):952--975, 2014.

\bibitem[BM14a]{BM2}
Arend Bayer and Emanuele Macr{\`{\i}}.
\newblock M{MP} for moduli of sheaves on {K}3s via wall-crossing: nef and
  movable cones, {L}agrangian fibrations.
\newblock {\em Invent. Math.}, 198(3):505--590, 2014.

\bibitem[BM14b]{BM1}
Arend Bayer and Emanuele Macr{\`{\i}}.
\newblock Projectivity and birational geometry of {B}ridgeland moduli spaces.
\newblock {\em J. Amer. Math. Soc.}, 27(3):707--752, 2014.

\bibitem[Eic74]{Eichler}
Martin Eichler.
\newblock {\em Quadratische {F}ormen und orthogonale {G}ruppen}.
\newblock Springer-Verlag, Berlin, 1974.
\newblock Zweite Auflage, Die Grundlehren der mathematischen Wissenschaften,
  Band 63.

\bibitem[GHS10]{GHS}
V.~Gritsenko, K.~Hulek, and G.~K. Sankaran.
\newblock Moduli spaces of irreducible symplectic manifolds.
\newblock {\em Compos. Math.}, 146(2):404--434, 2010.

\bibitem[HHT12]{HHT3}
David Harvey, Brendan Hassett, and Yuri Tschinkel.
\newblock Characterizing projective spaces on deformations of {H}ilbert schemes
  of {K}3 surfaces.
\newblock {\em Comm. Pure Appl. Math.}, 65(2):264--286, 2012.

\bibitem[HT09]{HT2}
Brendan Hassett and Yuri Tschinkel.
\newblock Moving and ample cones of holomorphic symplectic fourfolds.
\newblock {\em Geom. Funct. Anal.}, 19(4):1065--1080, 2009.

\bibitem[HT10a]{HTJus}
Brendan Hassett and Yuri Tschinkel.
\newblock Flops on holomorphic symplectic fourfolds and determinantal cubic
  hypersurfaces.
\newblock {\em J. Inst. Math. Jussieu}, 9(1):125--153, 2010.

\bibitem[HT10b]{HT4}
Brendan Hassett and Yuri Tschinkel.
\newblock Intersection numbers of extremal rays on holomorphic symplectic
  varieties.
\newblock {\em Asian J. Math.}, 14(3):303--322, 2010.

\bibitem[Huy99]{HuyBasic}
Daniel Huybrechts.
\newblock Compact hyper-{K}\"ahler manifolds: basic results.
\newblock {\em Invent. Math.}, 135(1):63--113, 1999.

\bibitem[Huy11]{huy}
Daniel Huybrechts.
\newblock A global {T}orelli theorem for hyperk\"ahler manifolds (after
  {M}.~{V}erbitsky), 2011.
\newblock S{\'e}minaire Bourbaki. Vol. 2010/2011, Exp. No. 1040.

\bibitem[HVA13]{HVA}
Brendan Hassett and Anthony V{\'a}rilly-Alvarado.
\newblock Failure of the {H}asse principle on general {$K3$} surfaces.
\newblock {\em J. Inst. Math. Jussieu}, 12(4):853--877, 2013.

\bibitem[HVAV11]{HVAV}
Brendan Hassett, Anthony V{\'a}rilly-Alvarado, and Patrick Varilly.
\newblock Transcendental obstructions to weak approximation on general {K}3
  surfaces.
\newblock {\em Adv. Math.}, 228(3):1377--1404, 2011.

\bibitem[HW60]{HW}
G.~H. Hardy and E.~M. Wright.
\newblock {\em An introduction to the theory of numbers}.
\newblock The Clarendon Press Oxford University Press, London, fourth edition,
  1960.

\bibitem[Mar08]{MarkJAG}
Eyal Markman.
\newblock On the monodromy of moduli spaces of sheaves on {$K3$} surfaces.
\newblock {\em J. Algebraic Geom.}, 17(1):29--99, 2008.

\bibitem[Mar11]{MarkSurv}
Eyal Markman.
\newblock A survey of {T}orelli and monodromy results for
  holomorphic-symplectic varieties.
\newblock In {\em Complex and differential geometry}, volume~8 of {\em Springer
  Proc. Math.}, pages 257--322. Springer, Heidelberg, 2011.

\bibitem[MSTVA14]{McKSaTaVA}
Kelly McKinnie, Justin Sawon, Sho Tanimoto, and Anthony V\'arilly-Alvarado.
\newblock Brauer groups on {$K3$} surfaces and arithmetic applications, 2014.
\newblock arXiv:1404.5460.

\bibitem[Muk06]{MukaiGenus13}
Shigeru Mukai.
\newblock Polarized {$K3$} surfaces of genus thirteen.
\newblock In {\em Moduli spaces and arithmetic geometry}, volume~45 of {\em
  Adv. Stud. Pure Math.}, pages 315--326. Math. Soc. Japan, Tokyo, 2006.

\bibitem[MY14]{MarkYosh}
Eyal Markman and Kota Yoshioka.
\newblock A proof of the {K}awamata-{M}orrison conjecture for holomorphic
  symplectic varieties of ${K}3^{[n]}$ or generalized {K}ummer deformation
  type, 2014.
\newblock arXiv:1402.2049.

\bibitem[O'G99]{ogrady}
Kieran~G. O'Grady.
\newblock Desingularized moduli spaces of sheaves on a {$K3$}.
\newblock {\em J. Reine Angew. Math.}, 512:49--117, 1999.

\bibitem[Ogu14]{oguiso}
K.~Oguiso.
\newblock Some aspects of explicit birational geometry inspired by complex
  dynamics, 2014.
\newblock arXiv:1404.2982.

\bibitem[Ver13]{verb}
Misha Verbitsky.
\newblock Mapping class group and a global {T}orelli theorem for hyperkahler
  manifolds.
\newblock {\em Duke Mathematical Journal}, 162(15):2929--2986, 2013.
\newblock arXiv:0908.4121.

\bibitem[Yos12]{Yosh}
Kota Yoshioka.
\newblock Bridgeland's stability and the positive cone of the moduli spaces of
  stable objects on an abelian surface, 2012.
\newblock arXiv:1206.4838.

\end{thebibliography}

\end{document}